\documentclass[12pt]{article}
\usepackage{a4wide}

\usepackage[a4paper, margin=2.3cm]{geometry}
\usepackage{amsmath,amssymb,amsthm}
\usepackage{color}
\usepackage{parskip}
\usepackage{parskip}
\usepackage{setspace}
\usepackage{graphicx}

\newtheorem{theorem}{Theorem}[section]
\newtheorem{remark}{Remark}[section]
\newtheorem{corollary}[theorem]{Corollary}
\newtheorem{lemma}[theorem]{Lemma}

\newtheorem{conjecture}[theorem]{Conjecture}
\newtheorem*{definition*}{Definition}

\def\RR{\mathcal{R}}
\def\S{\mathcal{S}}
\newcommand{\I}{\mathcal{I}}

\begin{document}
\title{Distribution of distances in positive characteristic}

\author{ Thang Pham \thanks{Department of Mathematics, ETH Switzerland. Email: phamanhthang.vnu@gmail.com} \and Vinh Le \thanks{Vietnam Institute of Educational Sciences. Email: vinhle@vnies.edu.vn}}
\date{}
\maketitle  
\begin{abstract}
Let $\mathbb{F}_q$ be an arbitrary finite field, and $\mathcal{E}$ be a point set in $\mathbb{F}_q^d$. Let $\Delta(\mathcal{E})$ be the set of distances determined by pairs of points in $\mathcal{E}$. By using Kloosterman sums, Iosevich and Rudnev (2007) proved that if $|\mathcal{E}|\ge 4q^{\frac{d+1}{2}}$ then $\Delta(\mathcal{E})=\mathbb{F}_q$. In general, this result is sharp in odd dimensional spaces over arbitrary finite fields. In this paper, we use the point-plane incidence bound due to Rudnev to prove that if $\mathcal{E}$ has Cartesian product structure in vector spaces over prime fields, then we can break the exponent $(d+1)/2$ and still cover all distances. We also show that the number of pairs of points in $\mathcal{E}$ of any given distance is close to its expected value. 
\end{abstract}
\section{Introduction}
Let $\mathcal{E}$ be a finite subset of $\mathbb{R}^d$ ($d \ge 2$), and $\Delta(\mathcal{E})$ be the distance set determined by $\mathcal{E}$. The Erd\H{o}s distinct distances problem is to find the best lower bound of the size of the distance set $\Delta(\mathcal{E})$ in terms of the size of the point set $\mathcal{E}$.

In the plane case, Erd\H{o}s \cite{er} conjectured that $|\Delta(\mathcal{E})| \gg |\mathcal{E}|/\sqrt{\log |\mathcal{E}|}$. This conjecture was proved up to logarithmic factor by Guth and Katz \cite{guth} in 2010. More precisely, they showed that $|\Delta(\mathcal{E})| \gg |\mathcal{E}|/\log |\mathcal{E}|$. In higher dimension cases, Erd\H{o}s \cite{er} also conjectured that  $|\Delta(\mathcal{E})| \gg |\mathcal{E}|^{2/d}$. Interested readers are referred to \cite{SV} for results on Erd\H{o}s distinct distances problem in three and higher dimensions. 

In this paper, we use the following notations: $X \ll Y$ means that there exists  some absolute constant $C_1>0$ such that $X \leq C_1Y$,  $X\sim Y$ means $Y\ll X\ll Y$, $X \gtrsim Y$ means $X\gg (\log Y)^{-C_2} Y$ for some absolute constant $C_2>0$, and $X \gtrsim_d Y$ mean $X \ge C_3(\log Y)^{-C_4} Y$ for some positive constants $C_3, C_4$ depending on $d$.

As a continuous analog of the Erd\H{o}s distinct distances problem,  Falconer \cite{fal} asked how large the Hausdorff dimension of $\mathcal{E} \subset \mathbb{R}^d$ needs to be to ensure that the Lebesgue measure of $\Delta(\mathcal{E})$ is positive. 
He conjectured that for any subset $\mathcal{E} \subset \mathbb{R}^d$ of the Hausdorff dimension greater than $d/2$ then $\mathcal{E}$ determines a distance set of a positive Lebesgue measure. This conjecture is still open in all dimensions. We refer readers to \cite{g-h, alex-fal} for recent updates on this conjecture. 

Let $\mathbb{F}_q$ be the finite field of order $q$, where $q$ is an odd prime power. Given two points $\mathbf{x} = (x_1, \ldots, x_d)$ and $\mathbf{y} = (y_1, \ldots, y_d)$ in $\mathbb{F}_q^d$, we denote the distance between $\mathbf{x}$ and $\mathbf{y}$ by
\[ || \mathbf{x} - \mathbf{y} || := (x_1 - y_1)^2 + \ldots + (x_d - y_d)^2.\]
Note that the distance function defined here is not a metric but it is invariant under translations and actions of the orthogonal group.

For a subset $\mathcal{E} \subset \mathbb{F}_q^d$, we denote the set of all distances determined by $\mathcal{E}$ by 
\[ \Delta(\mathcal{E}): = \{|| \mathbf{x} - \mathbf{y}|| : \mathbf{x}, \mathbf{y} \in \mathcal{E} \}.\]

The finite field analogue of the Erd\H{o}s distinct distances problem was first studied by Bourgain, Katz, and Tao in 2003 \cite{bkt}. More precisely, they proved that in the prime field $\mathbb{F}_p$ with $p \equiv 3$ mod 4, for any subset $\mathcal{E} \subset \mathbb{F}_p^2$ of the cardinality $|\mathcal{E}| = p^{\alpha}$, $0 < \alpha < 2$, then $|\Delta(\mathcal{E})| \gg |\mathcal{E}|^{\frac{1}{2} + \epsilon}$ for some $\epsilon = \epsilon(\alpha) > 0$. 

Note that the condition $p \equiv 3 \mod 4$ in Bourgain, Katz, and Tao's result is necessary, since if $p \equiv 1$ mod 4, then there exists $i \in \mathbb{F}_p$ such that $i^2 = -1$. By taking $\mathcal{E} = \{(x, i x): x \in \mathbb{F}_p\}$, we have $|\mathcal{E}| = p$ and $\Delta(\mathcal{E}) = \{0\}$. 

In the setting of arbitrary finite fields $\mathbb{F}_q$, Iosevich and Rudnev \cite{ir} showed that Bourgain, Katz, and Tao's result does not hold. For example, assume that $q = p^2$, one can take $\mathcal{E} = \mathbb{F}_p^2$ then $\Delta(\mathcal{E}) = \mathbb{F}_p$ or $|\Delta(\mathcal{E})| = |\mathcal{E}|^{1/2}$. Thus, Iosevich and Rudnev reformulated the problem in the spirit of the Falconer distance conjecture over the Euclidean spaces. More precisely, they asked for a subset $\mathcal{E} \subset \mathbb{F}_q^d$, how large does $|\mathcal{E}|$ need to be to ensure that $\Delta(\mathcal{E})$ covers the whole field or at least a positive proportion of all elements of the field? 
 
Using Fourier analytic methods, Iosevich and Rudnev \cite{ir} proved that for any point set $\mathcal{E} \subset \mathbb{F}_q^d$ with the cardinality $|\mathcal{E}| \ge 4q^{(d+1)/2}$ then $\Delta(\mathcal{E})=\mathbb{F}_q$. Hart, Iosevich, Koh, and Rudnev \cite{hart} showed that, in general, the exponent $(d+1)/2$ cannot be improved when $d$ is odd, even if we only want to cover a positive proportion of all the distances. In even dimensional cases, it has been conjectured that the exponent $(d+1)/2$ can be improved to $d/2$, which is in line with the Falconer distance conjecture in the Euclidean space. 

In the plane case, Bennett, Hart, Iosevich, Pakianathan, and Rudnev \cite{BHIRP} proved that if $\mathcal{E} \subset \mathbb{F}_q^2$ of cardinality $|\mathcal{E}| \ge q^{4/3}$, then $\Delta(\mathcal{E})$ covers a positive proportion of all distances. Murphy and Petridis \cite{mur} showed that there are infinite subsets of $\mathbb{F}_q^2$ of size $q^{4/3}$ whose distance sets do not cover the whole field $\mathbb{F}_q$. It is not known whether there exist a small $c>0$ and a set $\mathcal{E}\subset \mathbb{F}_q^2$ with $|\mathcal{E}|\ge cq^{3/2}$ such that $\Delta(\mathcal{E})\ne \mathbb{F}_q$. We refer the interested reader to \cite[Theorem 2.7]{hart} for a construction in odd dimensional spaces.

Chapman et al. \cite{CEHIK10} broke the exponent $\frac{d+1}{2}$ to $\frac{d^2}{2d-1}$ under the additional assumption that the set $\mathcal{E}$ has Cartesian product structure. However, in this case, they can cover only a positive proportion of all distances. In the setting of prime fields, it has been proved in \cite{P} that for $A\subset \mathbb{F}_p$, we have $|\Delta(A^d)|\ge \frac{1}{c}\cdot \min \{|A|^{2-\frac{1}{2^{d-2}}} , p\}$ with $c=2^{\frac{2^{d-1}-1}{2^{d-2}}}$. Therefore, $|\Delta(A^d)|\ge \frac{p}{c}$ under the condition $|A|\ge p^{\frac{2^{d-2}}{2^{d-1}-1}}$. However, this result again only gives us a positive proportion of all distances, and does not tell us the number of pairs of any given distance. 

In this paper, we will show that if $\mathcal{E} \subset \mathbb{F}_p^d$ has Cartesian product structure, we can break the exponent $\frac{d+1}{2}$ due to Iosevich and Rudnev \cite{ir} and still cover all possible distances. Our main tool is the point-plane incidence bound due to Rudnev \cite{R}.

Our first result is for odd dimensional cases. 
\bigskip
\begin{theorem}\label{thm1}
Let $\mathbb{F}_p$ be a prime field, and $A$  be a set in $\mathbb{F}_p$. For an integer $d\ge 3$, suppose the set $A^{2d+1}\subset \mathbb{F}_p^{2d+1}$ satisfies \[|A^{2d+1}|\gtrsim_d p^{\frac{2d+2}{2}-\frac{3\cdot 2^{d-2}-d-1}{3\cdot 2^{d-1}-1}},\]
then we have 
\begin{itemize}
\item The distance set covers all elements in $\mathbb{F}_p$, namely,
\[\Delta(A^{2d+1})=\underbrace{(A-A)^2 + \dots + (A-A)^2}_{2d+1~\textup{terms}}=\mathbb{F}_p.\]
\item In addition, the number of pairs $(\mathbf{x}, \mathbf{y})\in A^{2d+1}\times A^{2d+1}$ satisfying $||\mathbf{x}-\mathbf{y}||=\lambda$ is $\sim p^{-1}|A|^{4d+2}$ for any $\lambda\in \mathbb{F}_p$. 
\end{itemize}
\end{theorem}
\bigskip
\begin{corollary}
For $A\subset \mathbb{F}_p$, suppose that $|A|\gtrsim p^{6/11}$, then we have 
\[\Delta(A^{7})=(A-A)^2+(A-A)^2+(A-A)^2+(A-A)^2+(A-A)^2+(A-A)^2+(A-A)^2=\mathbb{F}_p.\]
\end{corollary}
\bigskip
Our second result is for  even dimensional cases. 
\bigskip
\begin{theorem}\label{thm2}
Let $\mathbb{F}_p$ be a prime field, and $A$  be a set in $\mathbb{F}_p$. For an integer $d\ge 3$, suppose the set $A^{2d}\subset \mathbb{F}_p^{2d}$ satisfies \[|A^{2d}|\gtrsim_d p^{\frac{2d+1}{2}-\frac{2^d-2d-1}{2^{d+1}-2}},\]
then we have 
\begin{itemize}
\item The distance set covers all elements in $\mathbb{F}_p$, namely, 
\[\Delta(A^{2d})=\underbrace{(A-A)^2 + \dots + (A-A)^2}_{2d~\textup{terms}}=\mathbb{F}_p.\]
\item In addition, the number of pairs $(\mathbf{x}, \mathbf{y})\in A^{2d}\times A^{2d}$ satisfying $||\mathbf{x}-\mathbf{y}||=\lambda$ is $\sim p^{-1}|A|^{4d}$ for any $\lambda\in \mathbb{F}_p$. 
\end{itemize}
\end{theorem}
\bigskip
\begin{corollary}
For $A\subset \mathbb{F}_p$, suppose that $|A|\gtrsim p^{4/7}$, then we have 
\[\Delta(A^{6})=(A-A)^2+(A-A)^2+(A-A)^2+(A-A)^2+(A-A)^2+(A-A)^2=\mathbb{F}_p.\]
\end{corollary}
\begin{remark}
In the setting of arbitrary finite fields $\mathbb{F}_q$, we can not break the exponent $(d+1)/2$, and still cover all distances with the method in this paper and the distance energy in \cite[Lemma $3.1$]{HP}. More precisely, for $A\subset \mathbb{F}_q$, one can follow the proofs of Theorems \ref{thm1} and \ref{thm2} to get the conditions $|A^{2d+1}|\gg q^{\frac{2d+2}{2}+\frac{1}{4d}}$ and $|A^{2d}|\gg q^{\frac{2d+1}{2}+\frac{1}{4d-2}}$ for odd and even dimensions, respectively. 
\end{remark}
\bigskip
\begin{remark}
The Cauchy-Davenport theorem states that for $X, Y\subset \mathbb{F}_p$, we have $|X+Y|\ge \min\{p, |X|+|Y|-1\}$. It is not hard to check that $\Delta(A^{2d})=\Delta(A^d)+\Delta(A^d)$. The Chapman et al. 's result \cite{CEHIK10} tells us that $|\Delta(A^d)|\ge p/2$ whenever $|A|\gg p^{\frac{d}{2d-1}}$. Therefore, one can apply the Cauchy-Davenport theorem to show that $|\Delta(A^{2d})|\ge p-1$ under the condition $|A|\ge p^\frac{d}{2d-1}$. However, our set $A^{2d}$ lies on the $2d$-dimensional space $\mathbb{F}_p^{2d}$, thus the exponent $\frac{d}{2d-1}$ is worse than the Iosevich-Rudnev's exponent $\frac{2d+1}{4d}$. The same happens for odd dimensional spaces. Note that the bound $|\Delta(A^d)|\ge \frac{1}{c}\cdot \min \{|A|^{2-\frac{1}{2^{d-2}}} , p\}$ with $c=2^{\frac{2^{d-1}-1}{2^{d-2}}}$ in \cite{P} is not suitable for this approach since the constant factor $1/c$ is too small. 
\end{remark}

Let $\mathbb{F}_q$ be an arbitrary finite field, and $\mathcal{E}\subset \mathbb{F}_q^d$. The product set of $\mathcal{E}$, denoted by $\Pi(\mathcal{E})$, is defined as follows:
\[\Pi(\mathcal{E}):=\{\mathbf{x}\cdot \mathbf{y}\colon \mathbf{x}, \mathbf{y}\in \mathcal{E}\}.\]
Using Fourier analysis, Hart and Iosevich \cite{HI} proved that if $|\mathcal{E}|\gg q^{\frac{d+1}{2}}$, then $\Pi(\mathcal{E}) \supseteq \mathbb{F}_q\setminus \{0\}$. Moreover, under the same condition on the size of $\mathcal{E}$, we have the number of pairs $(\mathbf{x}, \mathbf{y})\in \mathcal{E}\times \mathcal{E}$ satisfying $\mathbf{x}\cdot \mathbf{y}=\lambda$ is $\sim q^{-1}|\mathcal{E}|^2$ for any $\lambda\ne 0$. If $\mathcal{E}$ has Cartesian product structure, i.e. $\mathcal{E}=A^d$ for some $A\subset \mathbb{F}_q$, then the condition $|\mathcal{E}|\gg q^{\frac{d+1}{2}}$ is equivalent with $|A|\gg q^{\frac{1}{2}+\frac{1}{2d}}$.

In the setting of prime fields $\mathbb{F}_p$, if $d=8$, Glibichuk and Konyagin \cite{Gli2} proved that for $A, B\subset \mathbb{F}_p$, if $|A|\lceil|B|/2\rceil\ge p$, then we have $8A\cdot B=\mathbb{F}_p$. This
result has been extended to arbitrary finite fields by Glibichuk and Rudnev \cite{Gl3}.

In this paper, using the techniques in the proof of Theorems \ref{thm2}, we are able to obtain the following theorem.
\bigskip
\begin{theorem}\label{tich}
For $A\subset \mathbb{F}_p$, suppose that $|A|\gtrsim p^{4/7}$, then we have 
\begin{itemize}
\item $6A\cdot A=A\cdot A+ A\cdot A+A\cdot A+A\cdot A+A\cdot A+A\cdot A=\mathbb{F}_p$.
\item For any $\lambda\in \mathbb{F}_p$, the number of pairs $(\mathbf{x}, \mathbf{y})\in A^6\times A^6$ such that $\mathbf{x}\cdot \mathbf{y}=\lambda$ is $\sim p^{-1}|A|^{12}$.
\end{itemize}
\end{theorem}
\bigskip
Note that our exponent $4/7$ improves the exponent $7/12$ of Hart and Iosevich \cite{HI} in the case $d=6$. The following is the conjecture due to Iosevich. 
\bigskip
\begin{conjecture}
Let $A$ be a set in $\mathbb{F}_p$, suppose that $|A|\gg p^{\frac{1}{2}+\epsilon}$ for any $\epsilon>0$, then we have 
\[A\cdot A+A\cdot A=\mathbb{F}_p, ~~(A-A)^2+(A-A)^2=\mathbb{F}_p.\]
\end{conjecture}
In the spirit of sum-product problems, it has been proved in \cite{P} that for $A\subset \mathbb{F}_p$, if $|A|\ll p^{\frac{1}{2}+\frac{1}{5\cdot 2^{d-1}-2}}$, $d\ge 2$,  then we have 
\[\max\left\lbrace |\Delta(A^d)|, |\Pi(A^d)|\right\rbrace \gg |A|^{2-\frac{1}{5\cdot 2^d-3}}.\]
Using our energies (Lemmas \ref{thm-long} and \ref{thm-longx} below), and the prime field analogue of Balog-Wooley decomposition energy due to Rudnev, Shkredov, and Stevens \cite{RR}, we are able to give the energy variant of this result. 
\bigskip
\begin{theorem}\label{last}
Let $d\ge 2$ be an integer, $A$ be a set in $\mathbb{F}_p$ with $|A|\ll p^{\frac{1}{2}+\frac{1}{5\cdot 2^{d-1}-2}}$. There exist two disjoint subsets $B$ and $C$ of $A$ such that $A=B\sqcup C$ and 
\[\max\left\lbrace E_d((B-B)^2), E_d(C\cdot C)\right\rbrace \ll d^4(\log|A|)^4 |A|^{4d-2+\frac{1}{5\cdot 2^{d-3}}},\]
where $E_d\left((B-B)^2\right)$ is the number of $4d$-tuples $\{(a_i, b_i, c_i, e_i)\}_{i=1}^d$ with $a_i, c_i, b_i, e_i\in B$
such that $(a_1-b_1)^2+\cdots+(a_d-b_d)^2=(c_1-e_1)^2+\cdots+(c_d-e_d)^2,$ and $E_d\left(C\cdot C\right)$ be the number of $4d$-tuples $\{(a_i, b_i, c_i, e_i)\}_{i=1}^d$ with $a_i, c_i, b_i, e_i\in C$
such that $a_1b_1+\cdots+a_db_d=c_1e_1+\cdots+c_de_d.$
\end{theorem}
\section{Preliminaries}
Let $E$ and $F$ be multi-sets in $\mathbb{F}_p^2$. We denote by $\overline{E}$ and $\overline{F}$ the sets of distinct elements in $E$ and $F$, respectively. For any multi-set $X$, we use the notation $|X|$ to denote the size of $X$. For $\lambda\in \mathbb{F}_p$, let $N(E, F, \lambda)$ be the number of pairs $\left((e_1, e_2), (f_1, f_2)\right)\in E\times F$ such that $e_1f_1+e_2+f_2=\lambda$. In the following lemma, we provide an upper bound and a lower bound of $N(E, F, \lambda)$ for any $\lambda\in \mathbb{F}_p$. Note that, this lemma is essentially the weighted version of the second listed author point-line incidences \cite{vinh-incidence} in the plane $\mathbb{F}_q^2$ (see also \cite[Lemma 14]{HLR}).

\begin{lemma}\label{fourier}
Let $E, F$ be multi-sets in $\mathbb{F}_p^2$. For any $\lambda\in \mathbb{F}_p$, we have
\[\left\vert N(E, F, \lambda)-\frac{|E||F|}{p}\right\vert \le p^{\frac{1}{2}}\left(\sum_{(e_1, e_2)\in \overline{E}}m_E((e_1, e_2))^2\sum_{(f_1, f_2)\in \overline{F}}m_F((f_1, f_2))^2\right)^{1/2},\]
where $m_X((a, b))$ is the multiplicity of $(a, b)$ in $X$ with $X\in \{E, F\}$. 
\end{lemma}
\begin{proof}
Let $\chi$ be a non-trivial additive character on $\mathbb{F}_p$. We have 
$$ N(E, F, \lambda)=\sum_{(e_1, e_2)\in \overline{E}, (f_1, f_2) \in\overline{F}} \frac{1}{p} m_E((e_1, e_2))m_F((f_1, f_2))\sum_{s \in {\Bbb F}_p} \chi(s\cdot (e_1f_1+e_2+f_2-\lambda)).$$  This gives us
$$N(E, F, \lambda)=\frac{|E||F|}{p}+L, $$ where
$$ L=\sum_{(e_1, e_2)\in \overline{E}, (f_1, f_2) \in \overline{F}}m_E((e_1, e_2))m_F((f_1, f_2)) \frac{1}{p} \sum_{s \not=0} \chi(s\cdot (e_1f_1+e_2+f_2-\lambda)).$$
If we view  $L$ as a sum in $(e_1, e_2)\in \overline{E}$, then we can apply the Cauchy-Schwarz inequality to derive the following:
\begin{align*}
L^2 &\leq \sum_{(e_1, e_2)\in \overline{E}}m_E((e_1, e_2))^2 \sum_{(e_1, e_2) \in  \mathbb{F}_p^{2}} \frac{1}{p^2} \sum_{s,s' \not=0}
\sum_{(f_1, f_2), (f_1', f_2')\in \overline{F}}m_F((f_1, f_2))m_F((f_1', f_2'))
\\&\cdot \chi(s\cdot (e_1f_1+e_2+f_2-\lambda)) \chi(s'\cdot(-e_1f_1'-e_2-f_2'+\lambda))\\
&=\sum_{(e_1, e_2)\in\overline{E}}m_E((e_1, e_2))^2 \frac{1}{p^2} \sum_{\substack{(e_1, e_2)\in \mathbb{F}_p^{2} \\ (f_1, f_2)\in \overline{F}\\ (f_1', f_2')\in \overline{F}\\ s, s'\ne 0}}m_F((f_1, f_2))m_F((f_1', f_2'))
\chi(e_1(sf_1-s'f_1'))\chi(e_2(s-s'))\\ &\cdot\chi(s(f_2-\lambda)-s'(f_2'-\lambda))\\
\end{align*}
\begin{align*}
&=\sum_{(e_1, e_2)\in \overline{E}}m_E((e_1, e_2))^2\sum_{\substack{s\ne 0\\(f_1, f_2)\in \overline{F}\\(f_1', f_2')\in \overline{F}\\ f_1=f_1'}}m_F((f_1, f_2))m_F((f_1', f_2'))\chi(s\cdot (f_2-f_2'))=I+II,
\end{align*}
where $I$ is the sum over all pairs $\left((f_1, f_2), (f_1, f_2')\right)$ with $f_2=f_2'$, and $II$ is the sum over all pairs $\left((f_1, f_2), (f_1', f_2')\right)$ with $f_2\ne f_2'$. 

It is not hard to check that if $f_2\ne f_2'$, then 
\[\sum_{s\ne 0}\chi(s\cdot (f_2-f_2'))=-1,\]
so $II<0$. Note that $|II|\le I$ since $L^2\ge 0$. 

On the other hand,  if $f_2=f_2'$, then 
\[\sum_{s\ne 0}\chi(s\cdot (f_2-f_2'))=p-1.\]
In other words, 

\[I\le p \sum_{(e_1, e_2)\in \overline{E}}m_E((e_1, e_2))^2\sum_{(f_1, f_2)\in \overline{F}}m_F((f_1, f_2))^2,\]
which implies that 
\[|L| \le \sqrt{I +II} \le p^{\frac{1}{2}}\left(\sum_{(e_1, e_2)\in \overline{E}}m_E((e_1, e_2))^2\sum_{(f_1, f_2)\in \overline{F}}m_F((f_1, f_2))^2\right)^{1/2}.\]
This completes the proof of the lemma.
\end{proof}
\bigskip
For $A\subset\mathbb{F}_p$, let $E_d\left((A-A)^2\right)$ be the number of $4d$-tuples $\{(a_i, b_i, c_i, e_i)\}_{i=1}^d$ with $a_i, c_i, b_i, e_i\in A$
such that 
\[(a_1-b_1)^2+\cdots+(a_d-b_d)^2=(c_1-e_1)^2+\cdots+(c_d-e_d)^2.\]

Similarly, let $E_d\left(A\cdot A\right)$ be the number of $4d$-tuples $\{(a_i, b_i, c_i, e_i)\}_{i=1}^d$ with $a_i, c_i, b_i, e_i\in A$
such that 
\[a_1b_1+\cdots+a_db_d=c_1e_1+\cdots+c_de_d.\]

In our next lemmas, we give recursive formulas for $E_d((A-A)^2)$ and $E_d(A\cdot A)$.
\bigskip
\begin{lemma}\label{thm-long}
For $A\subset \mathbb{F}_p$, we have
\[E_d\left((A-A)^2\right)\le C d^2(\log |A|)^2\left( \frac{|A|^{4d}}{p}+|A|^{2d+1}\sqrt{E_{d-1}\left((A-A)^2\right)}\right),\]
for some positive constant $C$.
\end{lemma}
The proof of this lemma will be given in the next section. The following result is a direct consequence, which tells us an upper bound of $E_d((A-A)^2)$.
\bigskip
\begin{corollary}\label{cor-hm1} Let $A$ be a set in $\mathbb{F}_p$. For $d\ge 2$, suppose that $|A|\gg (d \log|A|) p^{1/2}$, then we have 
\[E_d\left((A-A)^2\right)\ll d^2 (\log|A|)^2 \frac{|A|^{4d}}{p}+d^4 (\log|A|)^4 |A|^{4d-2+\frac{1}{2^{d-1}}}.\]
\end{corollary}
\begin{proof}
We prove by induction on $d$ that  
\[E_d\left((A-A)^2\right)\le 2C^2d^2 (\log|A|)^2 \frac{|A|^{4d}}{p}+2C^2d^4 (\log|A|)^4 |A|^{4d-2+\frac{1}{2^{d-1}}},\]
whenever $|A|\ge \sqrt{2}C (d\log |A|)p^{1/2}$, where the constant $C$ comes from Lemma \ref{thm-long}. 

The base case $d=2$ follows directly from Lemma \ref{thm-long} by using the trivial upper bound $|A|^3$ of  $E_1((A-A)^2)$. 
 
Suppose the statement holds for any $d-1\ge 2$, we now prove that it also holds for $d$. Indeed, by induction hypothesis, we have 
\begin{align}\label{eq1}E_{d-1}\left((A-A)^2\right)&\le 2C^2(d-1)^2(\log|A|)^2 \frac{|A|^{4(d-1)}}{p}+ 2C^2(d-1)^4 (\log|A|)^4 |A|^{4d-6+\frac{1}{2^{d-2}}}\\
&\le 2C^2d^2(\log|A|)^2 \frac{|A|^{4(d-1)}}{p}+ 2C^2d^4 (\log|A|)^4 |A|^{4d-6+\frac{1}{2^{d-2}}}.\nonumber
\end{align}
On the other hand, it follows from Lemma \ref{thm-long} that 
\begin{equation}\label{eq2}E_d\left((A-A)^2\right)\le Cd^2(\log|A|)^2 \left( \frac{|A|^{4d}}{p}+|A|^{2d+1}\sqrt{E_{d-1}\left((A-A)^2\right)}\right).\end{equation}
Putting (\ref{eq1}) and (\ref{eq2}) together, we obtain 
\begin{eqnarray*}
E_d\left((A-A)^2\right) &\le&  Cd^2(\log|A|)^2\frac{|A|^{4d}}{p} \\ 
&  & + \sqrt{2}C^{2}d^2(\log|A|)^2 |A|^{2d+1} \left( d \log|A|\frac{|A|^{2(d-1)}}{p^{1/2}} + d^2(\log|A|)^2|A|^{2d-3 + \frac{1}{2^{d-1}}} \right).
\end{eqnarray*}
Since $ \sqrt{2}C(d \log|A|) p^{1/2} \le |A|$, we have
\[ \sqrt{2}C^{2}d \log|A|\frac{|A|^{2(d-1)}}{p^{1/2}} \le  C\frac{|A|^{2d-1}}{p}. \]
This implies that
\[E_d\left((A-A)^2\right)\le 2Cd^2 (\log|A|)^2 \frac{|A|^{4d}}{p}+2C^{2}d^4 (\log|A|)^4 |A|^{4d-2+\frac{1}{2^{d-1}}},\]
completing the proof of the corollary.
\end{proof}
Similarly, for the case of product sets, we have
\bigskip
\begin{lemma}\label{thm-longx}
For $A\subset \mathbb{F}_p$, we have
\[E_d\left(A\cdot A\right)\le Cd^2(\log |A|)^2\left( \frac{|A|^{4d}}{p}+|A|^{2d+1}\sqrt{E_{d-1}\left(A\cdot A\right)} \right),\]
for some positive constant $C$.
\end{lemma}
\begin{proof}
The proof of this lemma is almost identical with that of Lemma \ref{thm-long}, so we omit it.
\end{proof}
\begin{corollary}\label{cor-hm1x} Let $A$ be a set in $\mathbb{F}_p$. For $d\ge 2$, suppose that $|A|\gg (d\log|A|)p^{1/2}$, then we have 
\[E_d\left(A\cdot A\right)\ll d^2 (\log|A|)^2 \frac{|A|^{4d}}{p}+d^4 (\log|A|)^4 |A|^{4d-2+\frac{1}{2^{d-1}}}.\]
\end{corollary}
\begin{proof}
The proof of Corollary \ref{cor-hm1x} is identical with that of Corollary \ref{cor-hm1} with Lemma \ref{thm-longx} in the place of Lemma \ref{thm-long}, thus we omit it. 
\end{proof}
\subsection{Proof of Lemma \ref{thm-long} }\label{sec1}

In the proof of Lemma \ref{thm-long}, we will use a point-plane incidence bound due to Rudnev \cite{R} and an argument in \cite[Theorem $32$]{sh13}. 

Let us first recall that if $\RR$ is a set of points in $\mathbb{F}_p^3$ and $\S$ is a set of planes in $\mathbb{F}_p^3$, then the number of incidences between $\RR$ and $\S$, denoted by $\I(\RR, \S)$, is the cardinality of the set $\{(r,s)\in \RR\times \S : r\in s\}$. The following is a version of Rudnev's point-plane incidence bound, which can be found in \cite{Z}.
\bigskip
\begin{theorem}[{\bf Rudnev}, \cite{R, Z}]\label{thm:rudnev}
Let $\RR$ be a set of points in $\mathbb{F}_p^3$ and $\S$ be a set of planes in $\mathbb{F}_p^3$, with $|\RR|\leq |\S|$.
Suppose that there is no line that contains $k$ points of $\RR$ and is contained in $k$ planes of $\S$.
Then
\[ \I(\RR,\S)\ll\frac{|\RR||\S|}{p}+ |\RR|^{1/2}|\S| +k|\S|.\]
\end{theorem}
\paragraph{Proof of Lemma \ref{thm-long}:}
We first have 
\[E_d\left((A-A)^2\right)=\sum_{t_1, t_2}r_{(d-1)(A-A)^2}(t_1)r_{(d-1)(A-A)^2}(t_2)f(t_1, t_2),\]
where $r_{(d-1)(A-A)^2}(t)$ is the number of $2(d-1)$ tuples $(a_1, \ldots, a_{d-1}, b_1, \ldots, b_{d-1})\in A^{2d-2}$ such that $(a_1-b_1)^2+\cdots+(a_{d-1}-b_{d-1})^2=t$, and $f(t_1, t_2)$ is 
the sum $\sum_{s}r_{(A-A)^2+t_1}(s)r_{(A-A)^2+t_2}(s)$. We now split the sum $E_d\left((A-A)^2\right)$ into intervals as follows. 
\[E_d\left((A-A)^2\right)\ll \sum_{i=1}^{L_1} \sum_{j=1}^{L_2}\sum_{t_1, t_2}f(t_1, t_2) r^{(i)}_{(d-1)(A-A)^2}(t_1)r^{(j)}_{(d-1)(A-A)^2}(t_2),\]
where $L _1\le \log(|A|^{2d-2}), L_2\le \log(|A|^{2d-2})$, $r_{(d-1)(A-A)^2}^{(i)}(t_1)$ is the restriction of the function $r_{(d-1)(A-A)^2}(x)$ on the set $P_i:=\{t \colon 2^i\le r_{(d-1)(A-A)^2}(t)< 2^{i+1}\}$.

Using the pigeon-hole principle two times, there exist sets $P_{i_0}$ and $P_{j_0}$ for some $i_0$ and $j_0$ such that 
\begin{eqnarray*}
E_d\left((A-A)^2\right) &\le& (2d-2)^2 (\log|A|)^2 \sum_{t_1, t_2}f(t_1, t_2) r^{(i_0)}_{(d-1)(A-A)^2}(t_1)r^{(j_0)}_{(d-1)(A-A)^2}(t_2) \\
&\ll& d^2 (\log |A|)^2 2^{i_0}2^{j_0}\sum_{t_1,t_2}f(t_1, t_2)P_{i_0}(t_1)P_{j_0}(t_2).
\end{eqnarray*}

One can check that the sum $\sum_{t_1,t_2}f(t_1, t_2)P_{i_0}(t_1)P_{j_0}(t_2)$ is equal to the number of incidences between the point set $\RR$ of points $(-2a, e, t_1+a^2-e^2)\in \mathbb{F}_p^3$ with $a\in A, e\in A, t_1\in P_{i_0}$, and the plane set $\S$ of planes in $\mathbb{F}_p^3$ defined by 
\[bX+2cY+Z=t_2-b^2+c^2,\] where $b\in A, c\in A$ and $t_2\in P_{j_0}$. Without loss of generality, we can assume that $|P_{i_0}|\le |P_{j_0}|$.

To apply Theorem \ref{thm:rudnev}, we need to bound the maximal number of collinear points in $\RR$. The projection of $\RR$ into the plane of the first two coordinates is the set $-2A\times A$, thus if a line is not vertical, then it contains at most $|A|$ points from $\RR$. If a line is vertical, then it contains at most $|P_{i_0}|$ points from $\RR$, but that line is not contained in any plane in $\S$. In other words, we can apply Theorem \ref{thm:rudnev} with $k=|A|$, and obtain the following
\begin{align*}
\sum_{t_1,t_2}f(t_1, t_2)P_{i_0}(t_1)P_{j_0}(t_2)&\ll \frac{|A|^4|P_{i_0}||P_{j_0}|}{p}+|A|^{3}|P_{i_0}|^{1/2}|P_{j_0}|+|A|^3|P_{j_0}|\\
&\ll  \frac{|A|^4|P_{i_0}||P_{j_0}|}{p}+|A|^{3}|P_{i_0}|^{1/2}|P_{j_0}|.
\end{align*}
We now fall into the following cases:

{\bf Case $1$:} If the first term dominates, we have 
\[\sum_{t_1,t_2}f(t_1, t_2)P_{i_0}(t_1)P_{j_0}(t_2)\ll \frac{|A|^4|P_{i_0}||P_{j_0}|}{p}.\]

{\bf Case $2$:} If the second term dominates, we have 
\[\sum_{t_1,t_2}f(t_1, t_2)P_{i_0}(t_1)P_{j_0}(t_2)\ll |A|^{3}|P_{i_0}|^{1/2}|P_{j_0}|.\]
Therefore, 

\begin{align*}
E_d\left((A-A)^2\right)&\ll d^2 (\log |A|)^2 2^{i_0}2^{j_0}\left(\frac{|A|^4|P_{i_0}||P_{j_0}|}{p}+|A|^{3}|P_{i_0}|^{1/2}|P_{j_0}|\right)\\
&\ll d^2 (\log |A|)^2 \left( \frac{|A|^{4d}}{p}+|A|^{2d+1}\sqrt{E_{d-1}\left((A-A)^2\right)} \right).
\end{align*}
where we have used the facts that 
\begin{itemize}
\item $2^{i_0}|P_{i_0}|^{1/2}\ll \sqrt{E_{d-1}\left((A-A)^2\right)}$, 
\item $2^{j_0}|P_{j_0}|\ll |A|^{2d-2}$,
\item $2^{i_0}|P_{i_0}|\ll |A|^{2d-2}$.
\end{itemize}
This completes the proof of the lemma.
$\hfill\square$

\section{Proof of Theorem \ref{thm1}}
\paragraph{Proof of Theorem \ref{thm1}:}
Let $\lambda$ be an arbitrary element in $\mathbb{F}_p$.  Let $E$ be the multi-set of points $(2x, x^2+(y_1-z_1)^2+\cdots+(y_d-z_d)^2)\in \mathbb{F}_p^2$ with $x, y_i, z_i\in A$, and $F$ be the multi-set of points $(-t, t^2+(u_1-v_1)^2+\cdots+(u_d-v_d)^2)\in \mathbb{F}_p^2$ with $t,  u_i, v_i\in A$.  We have $|E|=|F|=|A|^{2d+1}$.

It follows from Lemma \ref{fourier} that 
\begin{equation}\label{eq00001}\left\vert N(E, F, \lambda)-\frac{|E||F|}{p}\right\vert \le p^{\frac{1}{2}}\left(\sum_{(e_1, e_2)\in \overline{E}}m_E((e_1, e_2))^2\sum_{(f_1, f_2)\in \overline{F}}m_F((f_1, f_2))^2\right)^{1/2}.\end{equation}

We observe that if $N(E, F, \lambda)$ is equal to the number of pairs $(\mathbf{x}, \mathbf{y})\in A^{2d+1}\times A^{2d+1}$ such that $||\mathbf{x}-\mathbf{y}||=\lambda$. 

From the setting of $E$ and $F$, it is not hard to see that 

\begin{equation}\label{eq00002}\sum_{(e_1, e_2)\in \overline{E}}m_E((e_1, e_2))^2 =|A|E_d((A-A)^2),~ \sum_{(f_1, f_2)\in \overline{F}}m_F((f_1, f_2))^2 =|A|E_d((A-A)^2).\end{equation}

Putting (\ref{eq00001}) and (\ref{eq00002}) together, we have 
\begin{equation}\label{eqx1}\left\vert N(E, F, \lambda)-\frac{|A|^{4d+2}}{p}\right\vert \le p^{\frac{1}{2}}|A|E_d((A-A)^2).\end{equation}
On the other hand, Corollary \ref{cor-hm1} gives us 
\begin{equation}\label{eqx0}
E_d\left((A-A)^2\right) \ll d^2 (\log|A|)^2 \frac{|A|^{4d}}{p}+d^4 (\log|A|)^4 |A|^{4d-2+\frac{1}{2^{d-1}}}.\end{equation}
Substituting (\ref{eqx0}) into (\ref{eqx1}), we obtain $N(E, F, \lambda)\sim |A|^{4d+2}p^{-1}$ whenever
\[|A^{2d+1}|\gtrsim_d p^{\frac{2d+2}{2}-\frac{3\cdot 2^{d-2}-d-1}{3\cdot 2^{d-1}-1}}.\]
Since $\lambda$ is arbitrary in $\mathbb{F}_p$, the theorem follows.  $\hfill\square$
\section{Proofs of Theorems \ref{thm2} and \ref{tich}}
The proof of Theorem \ref{thm2} is similar to that of Theorem \ref{thm1}, but we need a higher dimensional version of Lemma \ref{fourier}. 

Let $E$ and $F$ be multi-sets in $\mathbb{F}_p^3$. For $\lambda\in \mathbb{F}_p$, let $N(E, F, \lambda)$ be the number of pairs $\left((e_1, e_2, e_3), (f_1, f_2, f_3)\right)\in E\times F$ such that $e_1f_1+e_2f_2+e_3+f_3=\lambda$. One can follow step by step the proof of Lemma \ref{fourier} to obtain the following. 
\bigskip
\begin{lemma}\label{fourier1}
Let $E, F$ be multi-sets in $\mathbb{F}_p^3$. For any $\lambda\in \mathbb{F}_p$, we have
\[\left\vert N(E, F, \lambda)-\frac{|E||F|}{p}\right\vert \le p\left(\sum_{(e_1, e_2, e_3)\in \overline{E}}m_E((e_1, e_2, e_3))^2\sum_{(f_1, f_2, f_3)\in \overline{F}}m_F((f_1, f_2, f_3))^2\right)^{1/2}.\]
\end{lemma}
We are now ready to prove Theorem \ref{thm2}.

\paragraph{Proof of Theorem \ref{thm2}:}
Let $\lambda$ be an arbitrary element in $\mathbb{F}_p$.  Let $E$ be the multi-set of points $(2x_1, 2x_2,  x_1^2+x_2^2+(y_1-z_1)^2+\cdots+(y_{d-1}-z_{d-1})^2)\in \mathbb{F}_p^3$ with $x_i, y_i, z_i\in A$, and $F$ be the multi-set of points $(-t_1, -t_2,  t_1^2+t_2^2+(u_1-v_1)^2+\cdots+(u_{d-1}-v_{d-1})^2)\in \mathbb{F}_p^3$ with $t_i,  u_i, v_i\in A$.  We have $|E|=|A|^{2d}$ and $|F|=|A|^{2d}$.

It follows from Lemma \ref{fourier1} that 
\begin{equation}\label{eq0000111}\left\vert N(E, F, \lambda)-\frac{|E||F|}{p}\right\vert \le p\left(\sum_{(e_1, e_2, e_3)\in \overline{E}}m_E((e_1, e_2, _3))^2\sum_{(f_1, f_2, f_3)\in \overline{F}}m_F((f_1, f_2, f_3))^2\right)^{1/2}.\end{equation}

We observe that if $N(E, F, \lambda)$ is equal to the number of pairs $(\mathbf{x}, \mathbf{y})\in A^{2d}\times A^{2d}$ such that $||\mathbf{x}-\mathbf{y}||=\lambda$. 

From the setting of $E$ and $F$, it is not hard to see that 

\begin{equation}\label{eq0000211}\sum_{(e_1, e_2, e_3)\in \overline{E}}m_E((e_1, e_2, e_3))^2 =|A|^2E_{d-1}((A-A)^2),~ \sum_{(f_1, f_2, f_3)\in \overline{F}}m_F((f_1, f_2, f_3))^2 =|A|^2E_{d-1}((A-A)^2).\end{equation}

Putting (\ref{eq0000111}) and (\ref{eq0000211}) together, we have 
\begin{equation}\label{eqx111}\left\vert N(E, F, \lambda)-\frac{|A|^{4d}}{p}\right\vert \le p|A|^2E_{d-1}((A-A)^2).\end{equation}
On the other hand, Corollary \ref{cor-hm1} gives us 
\begin{equation}\label{eqx011}
E_{d-1}\left((A-A)^2\right) \ll d^2 (\log|A|)^2 \frac{|A|^{4d-4}}{p}+d^4 (\log|A|)^4 |A|^{4d-6+\frac{1}{2^{d-2}}}.\end{equation}
Substituting (\ref{eqx011}) into (\ref{eqx111}), we obtain $N(E, F, \lambda)\sim |A|^{4d}p^{-1}$ whenever
\[|A^{2d}|\gtrsim_d p^{\frac{2d+1}{2}-\frac{2^d-2d-1}{2^{d+1}-2}}.\]
Since $\lambda$ is arbitrary in $\mathbb{F}_p$,  the theorem follows. $\hfill\square$
\paragraph{Proof of Theorem \ref{tich}:} The proof of Theorem \ref{tich} is similar to that of Theorem \ref{thm2} with Corollary \ref{cor-hm1x} in the place of Corollary \ref{cor-hm1}. $\hfill\square$
\section{Proof of Theorem \ref{last}}
Let us first recall the prime field analogue of Balog-Wooley decomposition energy due to Rudnev, Shkredov, Stevens \cite{RR}. 
\bigskip
\begin{theorem}[\cite{RR}]\label{big}
Let $A$ be a set in $\mathbb{F}_p$ with $|A|\le p^{5/8}$. There exist two disjoint subsets $B$ and $C$ of $A$ such that $A=B\sqcup C$ and 
\[\max\{E^+(B), E^\times (C)\}\lesssim |A|^{14/5},\]
where $E^+(B)=|\{(a, b, c, d)\in B^4\colon a+b=c+d\}|$, and $E^\times (C)=|\{(a, b, c, d)\in C^4\colon ab=cd\}|$.
\end{theorem}
We refer the interested reader to \cite{balog} for the orginal result over $\mathbb{R}$. The most up to date bound for this result over $\mathbb{R}$ is due to Shakan \cite{shakan19}.

The following is another corollary of Lemma \ref{thm-long}. 
\bigskip
\begin{corollary}\label{hp1}
Let $A$ be a set in $\mathbb{F}_p$, and $B$ be a subset of $A$. For an integer $d\ge 2$, 
suppose that $|A|\ll p^{\frac{1}{2}+\frac{1}{5\cdot 2^{d-1}-2}}$ and $E^+(B)\lesssim |A|^{14/5}$, then we have 
\[E_d((B-B)^2)\ll d^4(\log|A|)^4 |A|^{4d-2+\frac{1}{5\cdot 2^{d-3}}}.\]
\end{corollary}
\begin{proof}
We prove by induction on $d$ that 
\[E_d((B-B)^2)\le 4C^2d^4(\log|A|)^4 |A|^{4d-2+\frac{1}{5\cdot 2^{d-3}}},\]
whenever $|A|\le (Cp)^{\frac{1}{2}+\frac{1}{5\cdot 2^{d-1}-2}}$, where the constant $C$ comes from Lemma \ref{thm-long}. 

The base case $d=2$ follows directly from Lemma \ref{thm-long} and the facts that $E_1((B-B)^2)\ll E^+(B)$ and $|B|\le |A|$. 

Suppose the corollary holds for $d-1\ge 2$, we now show that it also holds for the case $d$. Indeed, it follows from Lemma \ref{thm-long} that 
\[E_d\left((B-B)^2\right)\le Cd^2 (\log|A|)^2\left( \frac{|B|^{4d}}{p}+|B|^{2d+1}\sqrt{E_{d-1}\left((B-B)^2\right)} \right).\]
On the other hand, by induction hypothesis, we have 
\[E_{d-1}((B-B)^2)\le 4C^2(d-1)^4(\log|A|)^4 |A|^{4d-6+\frac{1}{5\cdot 2^{d-4}}}\le 4C^2d^4(\log|A|)^4 |A|^{4d-6+\frac{1}{5\cdot 2^{d-4}}}.\]
Thus, using the fact that $|B|\le |A|$, we obtain 
\[E_d((B-B)^2)\le d^2(\log|A|)^2 \left( C\frac{|A|^{4d}}{p}+2C^2d^2(\log|A|)^2|A|^{4d-2+\frac{1}{5\cdot 2^{d-3}}}\right)\le 4C^2d^4 (\log|A|)^4 |A|^{4d-2+\frac{1}{5\cdot 2^{d-3}}},\]
whenever $|A|\le (Cp)^{\frac{1}{2}+\frac{1}{5\cdot 2^{d-1}-2}}$.
\end{proof}
Using the same argument, we also have another corollary of Lemma \ref{thm-longx}.
\bigskip
\begin{corollary}\label{hp2}
Let $A$ be a set in $\mathbb{F}_p$, and $C$ be a subset of $A$. For an integer $d\ge 2$, suppose that $|A|\ll p^{\frac{1}{2}+\frac{1}{5\cdot 2^{d-1}-2}}$ and $E^\times(C)\lesssim |A|^{14/5}$, then we have 
\[E_d(C\cdot C)\ll d^4(\log|A|)^4 |A|^{4d-2+\frac{1}{5\cdot 2^{d-3}}}.\]
\end{corollary}
\bigskip
We are now ready to prove Theorem \ref{last}. 
\paragraph{Proof of Theorem \ref{last}:}
It follows from Theorem \ref{big} that there exist two disjoint subsets $B$ and $C$ of $A$ such that $A=B \sqcup C$ and $\max\{E^+(B), E^\times (C)\}\lesssim |A|^{14/5}$. One now can apply Corollaries \ref{hp1} and \ref{hp2} to derive 
\[\max\left\lbrace E_d((B-B)^2), E_d(C\cdot C)\right\rbrace \ll d^4(\log|A|)^4|A|^{4d-2+\frac{1}{5\cdot 2^{d-3}}}.\]
This completes the proof of the theorem. $\hfill\square$
\paragraph{Acknowledgments:} 
The authors are  grateful to the referee for useful comments and corrections. The first listed author was supported by Swiss National Science Foundation
grant P400P2-183916. The second listed author was supported by the National Foundation for Science and Technology Development Project. 101.99-2019.318.

\end{document}